\documentclass[11pt]{amsproc}
\usepackage{ytableau,amsthm,amsmath,amssymb,amsaddr,graphicx}
\newtheorem*{main}{Main Theorem}
\newtheorem{theorem}{Theorem}
\newtheorem{lemma}[theorem]{Lemma}
\theoremstyle{definition}
\newtheorem{definition}[theorem]{Definition}
\theoremstyle{remark}
\newtheorem{remark}[theorem]{Remark}
\DeclareMathOperator{\Ind}{Ind}
\DeclareMathOperator{\ch}{ch}
\newcommand\imm{\preccurlyeq}
\newcommand\CC{\mathbf C}
\usepackage{hyperref}
\title[Elementwise invariant vectors]{Existence of Elementwise Invariant Vectors in Representations of Symmetric Groups}
\author[]{Amrutha P}
\address{Chennai Mathematical Institute, Siruseri}
\email{\small{amruthap@cmi.ac.in, amri@imsc.res.in, velmurugan@imsc.res.in}}
\author[Amrutha P, Amritanshu Prasad and Velmurugan S]{Amritanshu Prasad and Velmurugan S}
\address{The Institute of Mathematical Sciences, Chennai}
\address{Homi Bhabha National Institute, Mumbai}

\begin{document}
\begin{abstract}
  We determine when a permutation with cycle type $\mu$ admits a non-zero invariant vector in the irreducible representation $V_\lambda$ of the symmetric group.
  We find that a majority of pairs $(\lambda,\mu)$ have this property, with only a few simple exceptions.
\end{abstract}

\maketitle
\section{Introduction}
\label{section:intro}
Let $(\rho_\lambda,V_\lambda)$ denote the irreducible representation of $S_n$ associated with a partition $\lambda$ of $n$. Let $w_\mu\in S_n$ denote a permutation with cycle type $\mu$ for a partition $\mu$ of $n$.
A vector $v\in V_\lambda$ is said to be an invariant vector for $w\in S_n$ if $\rho_\lambda(w)v=v$.
\label{sec:statement}
\begin{main}
  \label{theorem:main}
  The only pairs of partitions $(\lambda,\mu)$ of a given integer $n$ such that $w_\mu$ does not admit a nonzero invariant vector in $V_\lambda$ are the following:
  \begin{enumerate}
  \item $\lambda=(1^n)$, $\mu$ is any partition of $n$ for which $w_\mu$ is odd,
  \item $\lambda=(n-1,1)$, $\mu=(n)$, $n\geq 2$,
  \item $\lambda=(2,1^{n-2})$, $\mu=(n)$, $n\geq 3$ is odd,
  \item $\lambda=(2^2,1^{n-4})$, $\mu=(n-2,2)$, $n\geq 5$ is odd,
  \item $\lambda=(2,2)$, $\mu=(3,1)$,
  \item $\lambda=(2^3)$, $\mu=(3,2,1)$,
  \item $\lambda=(2^4)$, $\mu=(5,3)$,
  \item $\lambda=(4,4)$, $\mu=(5,3)$,
  \item $\lambda=(2^5)$, $\mu=(5,3,2)$.
  \end{enumerate}
\end{main}

In the special case where $\mu=(n)$, the following characterisation was obtained by Joshua Swanson~\cite{swanson} (see also \cite{MR4328100}), proving a conjecture of Sheila Sundaram~\cite[Remark~4.8]{MR3855421}.
\begin{theorem}
  \label{theorem:swanson}
  For every positive integer $n$ and partition $\lambda$ of $n$, $w_{(n)}$ admits a non-zero invariant vector in $V_\lambda$ except in the following cases:
  \begin{enumerate}
  \item $\lambda=(n-1,1)$,
  \item $\lambda=(1^n)$ when $n$ is even,
  \item $\lambda=(2,1^{n-2})$ when $n\geq 3$ is odd.
  \end{enumerate}
\end{theorem}
We use the above result and the Littlewood-Richardson rule to obtain our main theorem.

For each partition $\mu$ of $n$, let $C_\mu$ denote the cyclic group generated by $w_\mu$.
The permutation $w_\mu$ admits a non-zero invariant vector in $V_\lambda$ if and only if the trivial representation of $C_\mu$ occurs in the restriction of $V_\lambda$ to $C_\mu$.
By Frobenius reciprocity, this is equivalent to $V_\lambda$ occurring in the induced representation $\Ind_{C_\mu}^{S_n} 1$.

A closely related problem was considered by Sheila Sundaram who proved the following result \cite[Theorem~5.1]{MR3805649}.
\begin{theorem}
  \label{theorem:sundaram}
  Let $\mu$ be a partition of an integer $n\neq 4,8$.
  Let $Z_\mu$ denote the centralizer of $w_\mu$ in $S_n$.
  Then $V_\lambda$ occurs in $\Ind_{Z_\mu}^{S_n}1$ for every partition $\lambda$ of $n$ if and only if $\mu$ has at least two parts, and all its parts are odd and distinct.
\end{theorem}
A conjugacy class in a group $G$ whose permutation representation for the conjugation action contains every irreducible representation of $G$ is called a \emph{global class}.
Thus her theorem characterizes the global classes of symmetric groups.

Since $C_\mu$ is a subgroup of $Z_\mu$, $\Ind_{Z_\mu}^{S_n}1$ is a subrepresentation of $\Ind_{C_\mu}^{S_n}1$.
Hence if $V_\lambda$ occurs in $\Ind_{Z_\mu}^{S_n}1$, it also occurs in $\Ind_{C_\mu}^{S_n}1$.
If the parts of $\mu$ are distinct and pairwise relatively prime, then $Z_\mu=C_\mu$.
In this case our result is consistent with hers.
Our proof strategy is also quite similar to hers.

Dipendra Prasad and Ravi Raghunathan~\cite{zbMATH07612832} proposed a partial order on automorphic representations called immersion.
Adapted to finite groups, it may be defined as follows. 
\begin{definition}
  Given representations $(\rho,V)$ and $(\sigma,W)$ of $G$, say that \emph{$V$ is immersed in $W$}, denoted $V \imm W$, if for every $g\in G$ and every $\lambda\in \CC$, the multiplicity of $\lambda$ as an eigenvalue of $\rho(g)$ does not exceed the multiplicity of $\lambda$ as an eigenvalue of $\sigma(g)$. 
\end{definition}
In particular, if $V$ is a subrepresentation of $W$, then $V\imm W$.
In the context of immersion, our main theorem implies the following result.
\begin{theorem}
  For every positive integer $n$, $V_{(n)}\imm V_\lambda$ unless $\lambda$ is one of the partitions of $n$ that occur in the statement of the main theorem.
  Similarly $V_{(1^n)}\imm V_\lambda$ unless the conjugate of $\lambda$ is one of the partitions that occurs in the statement of the main theorem.
\end{theorem}
\begin{proof}
The first assertion is a direct consequence of our main theorem.
The second follows from the first using the observation that if $\gamma:G\to \CC^*$ is a multiplicative character, then $V\imm W$ if and only if $V\otimes \gamma\imm W\otimes \gamma$.
Taking $\gamma$ to be the sign character of $S_n$, we get $V_{(n)}\imm W_\lambda$ if and only if $V_{(1^n)}\imm W_{\lambda'}$, where $\lambda'$ is the partition conjugate to $\lambda$.
\end{proof}
\subsection*{Supporting Code}
Some steps in the proof involve direct calculations using the Sage Mathematical Software system \cite{sagemath}.
Code for carrying out these calculations can be downloaded from:\\
\href{https://www.imsc.res.in/~amri/invariant_vectors/}{https://www.imsc.res.in/\~{}amri/invariant\_vectors/}

\section{Proof of the Main Theorem}
\label{section:proof}
Following \cite[Section I.2]{MR3443860}, let $\Lambda$ denote the ring of symmetric functions in infinitely many variables $x_1,x_2,\dotsc$ with integer coefficients.
For each partition $\lambda$, let $s_\lambda$ denote the Schur function corresponding to $\lambda$.
Given $f,g\in \Lambda$, we say that $f\geq g$ if $f-g$ is a non-negative integer combination of Schur functions.

For each representation $V$ of $S_n$, let $\ch V$ denote the Frobenius characteristic of $V$, which is a homogeneous symmetric function of degree $n$ \cite[Section~I.7]{MR3443860}.
Since $\ch V_\lambda = s_\lambda$, $V_\lambda$ occurs in $\Ind_{C_\mu}^{S_n}1$ if and only if
\begin{equation}
  \label{eq:positivity}
  \ch \Ind_{C_\mu}^{S_n}1 \geq s_\lambda.
\end{equation}

Define
$$
f_\mu = \ch \Ind_{C_\mu}^{S_n} 1.
$$
Let $S_\mu=S_{\mu_1}\times \dotsb \times S_{\mu_k}$ be the Young subgroup corresponding to $\mu$.
Let $D_\mu$ denote the subgroup of $S_\mu$ generated by the cycles of $w_\mu$.
Thus $D_\mu$ is a product of cyclic groups of orders $\mu_1,\mu_2,\ldots, \mu_k$.
Using induction in stages,
$$f_\mu=\ch \Ind_{S_\mu}^{S_n}\Ind_{D_\mu}^{S_\mu}\Ind_{C_\mu}^{D_\mu}1.$$
Therefore
\begin{equation}
  \label{eq:fmugeq}
  f_\mu \geq \ch \Ind_{S_\mu}^{S_n}\Ind_{D_\mu}^{S_\mu} 1 = \prod_{i=1}^k f_{(\mu_i)}.
\end{equation}
Swanson's theorem (Theorem~\ref{theorem:swanson}) tells us that $f_{(n)}\geq s_\lambda$ for most partitions $\lambda$ of $n$.
We will use this fact, together with the inequality \eqref{eq:fmugeq}, to establish \eqref{eq:positivity} in most cases.
This will be carried out by using the Littlewood-Richardson rule.
Recall that Littlewood-Richardson coefficients are defined by
\begin{displaymath}
  s_\alpha s_\beta = \sum_\lambda c^\lambda_{\alpha\beta} s_\lambda.
\end{displaymath}
The Littlewood-Richardson rule \cite[Section I.9]{MR3443860} asserts that  $c^\lambda_{\alpha\beta}$ is the number of LR-tableaux of shape $\lambda/\alpha$ and weight $\beta$. Recall that an LR-tableau is a semistandard skew-tableau whose reverse row reading word is a lattice permutation.
\begin{lemma}
  \label{lemma:alphabeta}
  For every partition $\lambda$ of $p+q$, and every partition $\alpha$ of $p$ that is contained in $\lambda$, there exists a partition $\beta$ of $q$ such that $s_\alpha s_\beta\geq s_\lambda$.
  \begin{proof}
    Let $T_{\lambda\alpha}$ denote the skew-tableau obtained by putting $i$ in the $i$th cell of each column of $\lambda/\alpha$.
    Let $\beta$ be the weight of $T_{\lambda\alpha}$.
    For example, if $\lambda = (5,4,4,1)$ and $\alpha = (3,2,1)$ then
    \ytableausetup{boxsize=1.2em}
    $$T = \begin{ytableau}
      *(yellow)  &*(yellow) &*(yellow) & 1 & 1 \\
      *(yellow)&*(yellow)  & 1 & 2\\
      *(yellow) & 1 & 2 & 3\\
      1 \\
    \end{ytableau},$$
    and $\beta$ is $(5,2,1)$.
    Since every $i+1$ occurs below an $i$, $T_{\lambda\alpha}$ is an LR-tableau.
    The Littlewood-Richardson rule implies that $s_\alpha s_\beta\geq s_\lambda$.
  \end{proof}
\end{lemma}
\begin{remark}
  \label{remark:most-dominant}
  Given $\alpha\subset\lambda$ (i.e., $\alpha$ is contained in $\lambda$), the construction of $\beta$ in the proof of the above lemma has the following property: if $\gamma\vdash q$ also satisfies $s_\alpha s_\gamma\geq s_\lambda$, then $\beta \geq \gamma$ in the dominance order.
  This is because, in a skew tableau of shape $\lambda/\alpha$, the maximum number of cells with entry $i$ cannot exceed the number of columns of length at least $i$.
\end{remark}
\begin{lemma}
  \label{lemma:choose-beta}
  Given integers $p\geq 2$, $q\geq 1$, and a partition $\lambda\vdash (p+q)$ different from $(1^{(p+q)})$, there exists a partition $\beta\vdash q$ such that $f_{(q)}\geq s_\beta$ and $\beta\subset\lambda$.
  \begin{proof}
    \ytableausetup{smalltableaux}
    
    We consider the following cases:
    \subsubsection*{Case 1: $\lambda\supset (q-1,1)$}
    Since $p\geq 2$, the skew shape $\lambda/(q-1,1)$ has at least two cells.
    If at least one of these cells lies in the first row of $\lambda$, then choose $\beta=(q)$.
    If at least one of these cells lies in the first column of $\lambda$, then choose $\beta=(q-2,1,1)$
    If neither of the above happens, then $\lambda/(q-1,1)$ has at least two cells in its second row.
    In this case $q-1\geq 3$.
    Choose $\beta=(q-2,2)$.
    \begin{figure}[h]
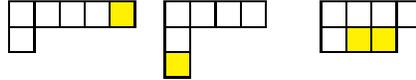

      \begin{displaymath}
        \ydiagram{4,1}*[*(yellow)]{4+1,1+0}\quad \ydiagram{4,1}*[*(yellow)]{4+0,1+0,1} \quad \ydiagram{4,1}*[*(yellow)]{4+0,1+2}
      \end{displaymath}
      \caption{Possible placements of two cells of $\lambda/(q-1,1)$}
      \label{fig:q-1.1}
    \end{figure}
    The possible placements of the cells of $\lambda/(q-1,1)$ are shown in Figure~\ref{fig:q-1.1}.
    In all these cases, Theorem~\ref{theorem:swanson} implies that $f_{(q)}\geq s_\beta$.
    \subsubsection*{Case 2: $\lambda\supset (1^q)$ and $q$ is even}
    Since $\lambda\neq (1^{p+q})$, the skew-shape $\lambda/(1^q)$ must contain at least one cell in the first row.
    Take $\beta = (2,1^{q-2})$.
    By Theorem~\ref{theorem:swanson}, $f_{(q)}\geq s_\beta$, since $q$ is even.
    \subsubsection*{Case 3: $\lambda\supset (2,1^{q-2})$ and $q$ is odd}
    If $\lambda/(2,1^{q-2})$ has a cell in its first row, take $\beta=(3,1^{q-3})$.
    If $\lambda/(2,1^{q-2})$ has a cell in its first column, take $\beta=(1^q)$.
    By Theorem~\ref{theorem:swanson}, $f_{(q)}\geq s_\beta$, since $q$ is odd.
    Otherwise the second column of $\lambda/(2,1^{q-2})$ must have at least two cells in its second column.
    In this case $q\geq 4$.
    Take $\beta=(2,2,1^{q-4})$.
    \begin{figure}[h]
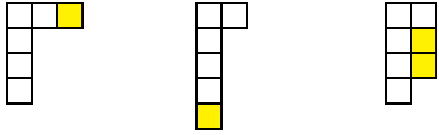

      \begin{displaymath}
        \ydiagram{2,1,1,1}*[*(yellow)]{2+1} \hspace{1.5cm} \ydiagram{2,1,1,1}*[*(yellow)]{2+0,0,0,0,1} \hspace{1.5cm}\ydiagram{2,1,1,1}*[*(yellow)]{2+0,1+1,1+1}
      \end{displaymath}
      \caption{Possible placements of cells of $\lambda/(2,1^{q-2})$.}
      \label{fig:2.1q-2}
    \end{figure}
    The possible placements of the cells of $\lambda/(2,1^{q-2})$ are shown in Figure~\ref{fig:2.1q-2}.
    \subsubsection*{All remaining $\lambda$:}
    Take $\beta$ to be any partition of $q$ that is contained in $\lambda$.
    Since $\lambda$ does not contain any of the exceptions of Theorem~\ref{theorem:swanson}, $f_{(q)}\geq s_\beta$.
  \end{proof}
\end{lemma}

\begin{definition}[Persistent partition]
  A partition $\mu$ of $n$ is said to be \emph{persistent} if  every irreducible representation of $S_n$, with the possible exception of the sign representation, admits a non-zero $w_\mu$-invariant vector.
\end{definition}
When $\mu$ is persistent,
\begin{displaymath}
  f_\mu \geq s_\alpha \text{ for every $\alpha\vdash n$, $\alpha\neq (1^n)$.}
\end{displaymath}
The main step in proving the main theorem is to show that $\mu$ is persistent for every integer partition $\mu$ except $\mu=(n)$ for all $n$ and $\mu=(n-2,2)$ for $n$ odd and sufficiently large.
We first consider the case where $\mu$ has two parts.
\begin{lemma}
\label{lemma:pq}
  If $\mu=(p,q)$ where $p\geq q\geq 4$, then $\mu$ is persistent.
\end{lemma}
\ytableausetup{boxsize=1.2em}
\begin{proof}
  We will assume that $p\geq 6$.
  For $p=4$ and $5$, direct computation (using Sage~\cite{sagemath}) establishes the lemma.
  By \eqref{eq:fmugeq}, to show that $f_{(p,q)}\geq s_\lambda$ it suffices to find partitions $\alpha$ and $\beta$ of $p$ and $q$ respectively, such that $f_{(p)}\geq s_\alpha$, $f_{(q)}\geq s_\beta$ and $s_\alpha s_\beta\geq s_\lambda$.

  By Lemma~\ref{lemma:choose-beta}, choose $\beta\subset \lambda$ such that $f_{(q)}\geq s_\beta$.
  Using Lemma~\ref{lemma:alphabeta} (with the roles of $\alpha$ and $\beta$ reversed) we may choose $\alpha\vdash p$ such that $s_\alpha s_\beta\geq s_\lambda$.
  If $f_{(p)}\geq s_\alpha$ we are done.
  Otherwise, by Theorem~\ref{theorem:swanson}, one of the following cases must occur:
  \subsubsection*{Case 1: $\alpha=(p-1,1)$}
  \begin{itemize}
  \item Suppose $\lambda_2\geq 2$.
    Replace $\alpha$ by $(p-2,2)$ and choose $\beta$ as in the proof of Lemma~\ref{lemma:alphabeta}.
    If $f_{(q)}\geq s_\beta$ we are done, otherwise by Theorem~\ref{theorem:swanson}, one of the following must occur.
    \begin{itemize}
    \item $\beta = (q-1,1)$.
      Since each column of $T_{\lambda\alpha}$ is filled with integers $1,2,\cdots$ in increasing order, $\lambda/\alpha$ has $q-2$ columns with one cell and one column with two cells.
      Either the first column, the third column, or the $(p-1)$st column can have two cells.

      If the first column of $\lambda/\alpha$ has two cells, then we may change $\beta$ to $(q-2,1,1)$ by constructing a skew-tableau as shown in the example below
      \begin{displaymath}
        \ydiagram{5,2}*[*(yellow) 2]{0,0,1}*[*(yellow) 3]{0,0,0,1}*[*(lime) 1]{5+2,2+3,1}
      \end{displaymath}
      Note that the reverse reading word will be a lattice permutation because $\lambda\supset (p-1,1)$, so there will be at least one cell in the first row of $\lambda/\alpha$.

      If the third column of $\lambda/\alpha$ has two cells, then we may change $\beta$ to $(q-2,2)$ as shown in the example below
      \begin{displaymath}
        \ydiagram{5,2}*[*(yellow) 1]{0,2+1}*[*(yellow) 2]{0,0,2+1}*[*(lime) 1]{5+2,2+3,1}*[*(lime) 2]{0,0,1+1}
      \end{displaymath}
      Note that the reverse reading word will be a lattice permutation because at least one $1$ occurs in the first row (since $\lambda\supset(p-1,1)$) and at least one $1$ occurs in the second row of this tableau (in the third column).

      If the $(p-1)$st column of $\lambda/\alpha$ has two cells, we may change $\alpha$ to $(p-3,3)$ (note that $p\geq 6$, so $\alpha$ is a partition).
      We may then take $\beta=(q-2,2)$ as shown in the example below.
      \begin{displaymath}
        \ydiagram{5,2}*[*(yellow) 1]{5+1}*[*(yellow) 2]{0,5+1}*[*(lime) 1]{6+2,2+3,1}
        \longrightarrow
        \ydiagram{4,3}*[*(yellow) 1]{5+1}*[*(yellow) 2]{0,5+1}*[*(lime) 1]{6+2,3+1,1}*[*(lime) 2]{0,4+1}*[*(lime) 1]{4+1}
      \end{displaymath}
    \item $\beta = (2,1^{q-2})$ with $q$ odd.
      This would mean that $\lambda/\alpha$ has two columns, having $1$ and $q-1$ cells respectively.
      But since $\lambda\supset (p-1,1)$ the $(p-1)$st cell in the first row lies in a column of length one.
      Since $q\geq 4$, the other column of $\lambda/\alpha$ has to be the first one.
      In particular, $\lambda/\alpha$ is a vertical strip, so we can replace beta with $(1^q)$.

    \item $\beta=(1^q)$ with $q$ even.
      In this case, $\lambda/\alpha$ has only one column.
      But $\lambda/\alpha$ contains the cell $(1,p-1)$ whose column cannot contain any other cells.
      It follows that this case cannot occur.
    \end{itemize}
  \item Suppose $\lambda_2=1$ and $\lambda_1\geq p$.
    Replace $\alpha$ by $(p)$ to make $f_{(p)}\geq s_\alpha$, and choose $\beta$ using the proof of Lemma~\ref{lemma:alphabeta}.
    By Theorem~\ref{theorem:swanson} $f_{(q)}\geq s_\beta$ unless one of the following cases occurs.
    \begin{itemize}
    \item $\beta = (q-1,1)$.
      In this case, $\lambda/\alpha$ has $q-1$ columns, of which one column has exactly two cells (and the others have only one cell).
      The column with two cells has to be the first column of $\lambda/\alpha$, since $\lambda$ is a hook.
      Incrementing its entries by $1$ allows us to replace $\beta$ by $(q-2,1,1)$.
    \item $\beta = (2,1^{q-2})$ with $q$ odd.
      Then $\lambda/\alpha$ has two columns, having $1$ and $q-1$ cells respectively.
      Since $\lambda$ is a hook, the column with one cell must lie in the first row and the column with $q-1$ cells has to be the first column.
      Incrementing the entries of the first column by $1$ allows us to replace $\beta$ by $(1^q)$.
    \item $\beta=(1^q)$ with $q$ even.
      All the cells of $\lambda/\alpha$ must lie in the first column.
      We may replace $\alpha$ by $(p-2,1,1)$ and $\beta$ by $(2,1^{q-2})$ as shown in the following example.
      \begin{displaymath}
        \ytableaushort{{}{}{}{}{}{},1,2,3,4}*[*(lime)]{0,1,1,1,1}\longrightarrow
        \ytableaushort{{}{}{}{}11,{},{},2,3}*[*(lime)]{4+2,0,0,1,1}
      \end{displaymath}
    \end{itemize}
  \item Suppose $\lambda_2=1$ and $\lambda_1=p-1$.
    In other words, $\lambda=(p-1,1^{q+1})$.
    Replace $\alpha$ by $(p-2,1,1)$.
    Replace $\beta$ by $(3,1^{q-3})$ if $q>4$ as shown in the example below.
    \begin{displaymath}
      \ytableaushort{{}{}{}{}{}{},1,2,3,4}*[*(lime)]{0,1,1,1,1}\longrightarrow
      \ytableaushort{{}{}{}{}11,{},{},1,2}*[*(lime)]{4+2,0,0,1,1}
    \end{displaymath}
    In case $q=4$, replace $\beta$ by $(2,1,1)$
  \end{itemize}

  \subsubsection*{Case 2: $\alpha = (2,1^{p-2})$ with $p$ odd}
  \begin{itemize}
  \item Suppose $\lambda_2\geq 2$.
    Change $\alpha$ to $(2,2,1^{p-4})$ so that $f_{(p)}\geq s_\alpha$.
    Choose $\beta$ as in the proof of Lemma~\ref{lemma:alphabeta}.
    If $f_{(q)}\geq s_\beta$ we are done.
    Otherwise, by Theorem~\ref{theorem:swanson}, one of the following cases must occur.
    \begin{itemize}
    \item $\beta=(q-1,1)$.
      In this case, $\lambda/\alpha$ has $q-1$ columns, with one column having two cells.
      The column with two cells has to be one of the first three columns.

      If the first or second column of $\lambda/\alpha$ has two cells, increment the entries of those two cells by $1$ to replace $\beta$ by $(q-2,1,1)$.
      Since $q\geq 4$, $\lambda/\alpha$ has at least one cell in the first row, so the resulting skew-tableau is an LR-tableau.

      Suppose the third column of $\lambda/\alpha$ has two cells.
      Since $\lambda\supset (2,1^{p-2})$, the first column of $\lambda/\alpha$ must have exactly one cell.
      Changing the entry of this cell from $1$ to $3$ gives us an LR-tableau of weight $(q-2,1,1)$.

    \item $\beta=(2,1^{q-2})$ with $q$ odd.
      In this case, $\lambda/\alpha$ has exactly two columns, having $1$ and $q-1$ cells respectively.
      Again, since $\lambda\supset (2,1^{p-2})$, $\lambda/\alpha$ has at least one cell in the first column.

      If the first column has $q-1$ cells, then $\lambda/\alpha$ has to be a vertical strip, and we can take $\beta=(1^q)$.

      If the first column has one cell and the second has $q-1$ cells, then the other column, which has at least three cells must be the second column.
      Replace $\alpha$ by $(2^3,1^{p-6})$ and take $\beta=(2,2,1^{q-4})$ as shown in the following example.
      \begin{displaymath}
        \ydiagram{2,2,1,1}*[*(lime) 1]{0,0,1+1,0,0+1}*[*(lime) 2]{0,0,0,1+1}*[*(lime) 3]{0,0,0,0,2}\longrightarrow
        \ydiagram{2,2,2}*[*(lime) 1]{0,0,0,2,0}*[*(lime) 2]{0,0,0,0,2}*[*(lime) 2]{0,0,0,0,2}
      \end{displaymath}
    \item $\beta=(1^q)$ with $q$ even.
      In this case, $\lambda/\alpha$ is a single column, which has to be the first column since $q\geq 4$.
      Replace $\alpha$ by $(1^p)$ and $\beta$ by $(2,1^{q-2})$ as shown in the following example.
      \begin{displaymath}
        \ydiagram{2,2,1,1}*[*(lime) 1]{0,0,0,0,0+1}*[*(lime) 2]{0,0,0,0,0,1}*[*(lime) 3]{0,0,0,0,0,0,1}*[*(lime) 4]{0,0,0,0,0,0,0,1}\longrightarrow
        \ydiagram{1,1,1,1,1,1}*[*(lime) 1]{1+1,0,0,0,0,0,1}*[*(lime) 2]{0,1+1,0,0,0,0,0,0}*[*(lime) 3]{0,1+1,0,0,0,0,0,1}
      \end{displaymath}
    \end{itemize}
  \item Suppose $\lambda_2=1$ (so that $\lambda$ is a hook) and $\lambda_1\geq 3$.
    Replace $\alpha$ by $(3,1^{p-3})$ and choose $\beta$ according to Lemma~\ref{lemma:alphabeta}.
    If $f_{(q)}\geq s_\beta$ we are done.
    Otherwise by Theorem~\ref{theorem:swanson} one of the following must occur.
    \begin{itemize}
    \item $\beta = (q-1,1)$. In this case, $\lambda$ has $q-1$ columns, one having two cells and the other having only one cell.
      Since $\lambda$ is a hook, only the first column of $\lambda/\alpha$ can have two cells.
      Incrementing the entries in these cells by one will allow us to change $\beta$ to $(q-2,1,1)$.
    \item $\beta = (2,1^{q-2})$ with $q$ odd.
      In this case, $\lambda/\alpha$ has two columns.
      Since $\lambda\supset (2,1^{p-2})$, $\lambda/\alpha$ has cells in the first column.
      Since $\lambda$ is a hook, the fourth column must have exactly one cell.
      Incrementing the entries of the first column by $1$ will allow us to replace $\beta$ by $(1^q)$.
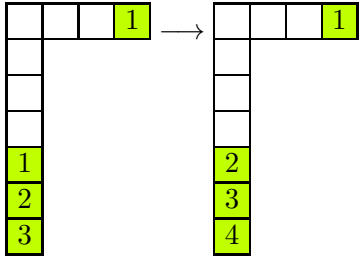
      \begin{displaymath}
        \ydiagram{3,1,1,1}*[*(lime) 1]{3+1,0,0,0,0+1}*[*(lime) 2]{0,0,0,0,0,1}*[*(lime) 3]{0,0,0,0,0,0,1}\longrightarrow
        \ydiagram{3,1,1,1}*[*(lime) 1]{3+1,0,0,0,0}*[*(lime) 2]{0,0,0,0,1,0}*[*(lime) 3]{0,0,0,0,0,1}*[*(lime) 4]{0,0,0,0,0,0,1}
      \end{displaymath}
    \item $\beta = (1^q)$, with $q$ even.
      In this case, $\lambda/\alpha$ is contained in the first column.
      We can replace $\alpha$ by $(1^p)$ (since $p$ is odd) and $\beta$ by $(2,1^{q-2})$ as shown in the example below.

      \begin{displaymath}
        \ydiagram{3,1,1,1}*[*(lime) 1]{0,0,0,0,0+1}*[*(lime) 2]{0,0,0,0,0,1}*[*(lime) 3]{0,0,0,0,0,0,1}*[*(lime) 4]{0,0,0,0,0,0,0,1}\longrightarrow
        \ydiagram{3,1,1,1,1,1}*[*(lime) 1]{1+2,0,0,0,0}*[*(lime) 2]{0,0,0,0,0,0,1}*[*(lime) 3]{0,0,0,0,0,0,0,1}
      \end{displaymath}

    \end{itemize}
  \item Suppose $\lambda_2=1$ and $\lambda_1=2$.
    In this case $\lambda$ must be $(2,1^{p+q-2})$. So we may replace $\alpha$ by $(1^p)$ and $\beta$ by $(2,1^{q-2})$ or $(1^q)$ depending on the parity of $q$.
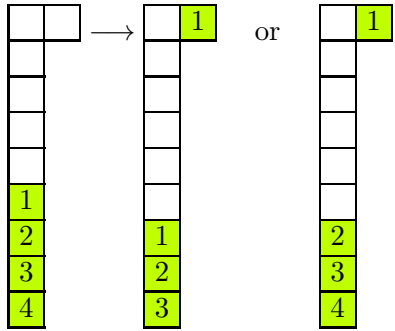
    \begin{displaymath}
      \ydiagram{2,1,1,1,1}*[*(lime) 1]{0,0,0,0,0,0+1}*[*(lime) 2]{0,0,0,0,0,0,1}*[*(lime) 3]{0,0,0,0,0,0,0,1}*[*(lime) 4]{0,0,0,0,0,0,0,0,1}\longrightarrow
      \ydiagram{1,1,1,1,1,1}*[*(lime) 1]{1+1,0,0,0,0,0,0+1}*[*(lime) 2]{0,0,0,0,0,0,0,1}*[*(lime) 3]{0,0,0,0,0,0,0,0,1}*[*(lime) 4]{0,0,0,0,0,0,0,0,1}
      \quad \text{ or }\quad
      \ydiagram{1,1,1,1,1,1}*[*(lime) 1]{1+1,0,0,0,0,0,0}*[*(lime) 2]{0,0,0,0,0,0,1}*[*(lime) 3]{0,0,0,0,0,0,0,1}*[*(lime) 4]{0,0,0,0,0,0,0,0,1}
    \end{displaymath}

  \end{itemize}
  \subsubsection*{Case 3: $\alpha = (1^p)$ with $p$ even}
  Since we have assumed that $\lambda\neq (1^{p+q})$, $\lambda_1>1$.
  Replace $\alpha$ by $(2,1^{p-2})$ and choose $\beta$ using the proof of Lemma~\ref{lemma:alphabeta}.
  If $f_{(q)}\geq s_\beta$ we are done.
  Otherwise, by Theorem~\ref{theorem:swanson}, one of the following cases must occur.
  \begin{itemize}
  \item $\beta = (q-1,1)$.
    Then $\lambda/\alpha$ has $q-1$ columns, one of which has two cells and the others have one cell each.
    The column with two cells has to be one of the first three.
    If the column with two cells is one of the first two columns, then incrementing its entries by $1$ will give an LR-tableau of shape $\lambda/\alpha$ and weight $(q-2,1,1)$.
    If the column with two cells is the third, then changing the entry in the first column (which exists since $\lambda\supset (1^p)$) from $1$ to $2$, we get an LR-tableau of shape $\lambda/\alpha$ and weight $(q-2,2)$.
  \item $\beta=(2,1^{q-2})$ with $q$ odd.
    In this case, $\lambda/\alpha$ has exactly two columns having $q-1$ and $1$ cells respectively.

    If the first column of $\lambda/\alpha$ has $q-1$ cells, then $\lambda/\alpha$ will be a vertical strip, so we can change $\beta$ to $(1^q)$.

    If the second column of $\lambda/\alpha$ has $q-1$ cells, replace $\alpha$ by $(2,2,1^{p-4})$ and $\beta$ by $(2,2,1^{q-4})$ as shown in the following example.
    \begin{displaymath}
      \ydiagram{2,1,1,1}*[*(lime) 1]{0,1+1,0,0,0+1}*[*(lime) 2]{0,0,1+1}*[*(lime) 3]{0,0,0,1+1}*[*(lime) 4]{0,0,0,0,1+1}\longrightarrow
      \ydiagram{2,2,1,1}*[*(lime) 1]{0,0,1+1,1,0}*[*(lime) 2]{0,0,0,1+1,1}*[*(lime) 3]{0,0,0,0,1+1}*[*(lime) 4]{0,0,0,0,1+1}
    \end{displaymath}
  \item $\beta=(1^q)$ with $q$ even.
    In this case, $\lambda/\alpha$ has only one column, which has to be the first column.
    So $\lambda=(2,1^{p+q-2})$.
    In this case it will not be true that $f_{(p)}f_{(q)}\geq s_\lambda$.
    Instead, we will show that $f_{(p,q)}\geq s_{(2,1^{p+q-2})}$ using a different strategy.

    The restriction of the sign representation of $S_p\times S_q$ to $C_{(p,q)}$ is the trivial representation ($w_{(p,q)}$ is an even permutation, since both $p$ and $q$ are even).
    Therefore, $\Ind_{C(p,q)}^{S_p\times S_q} 1$ contains the sign representation of $S_p\times S_q$.
    It follows that 
    \begin{displaymath}
      f_{(p,q)} = \ch \Ind_{C_{(p,q)}}^{S_{p+q}}1 \geq s_{(1^p)}s_{(1^q)} \geq s_{(2,1^{p+q-2})}.
    \end{displaymath}
  \end{itemize}
  This completes the proof of the lemma.
\end{proof}
\begin{lemma}
  \label{lemma:alphabeta1}
  We have
  \begin{enumerate}
  \item The partition $(p,1)$ is persistent except when $p=3$.
  \item The partition $(p,2)$ is persistent when p is even.
    When p is odd, $f_{(p,2)}\geq s_\lambda$ for all partitions $\lambda$ of $p+2$, except $\lambda=(2,2,1^{p-2})$ and $(1^{p+2})$.
  \item The partition $(p,3)$ is persistent if $p\geq 6$.
  \end{enumerate}
\end{lemma}
\begin{proof}
  In order to prove (1), assume $p>3$ (for $p\leq 3$ use direct calculation) and choose $\alpha\vdash p$ such that $s_\alpha s_\beta\geq s_\lambda$ using Lemma~\ref{lemma:alphabeta} with the roles of $\alpha$ and $\beta$ interchanged.
  If $f_{(p)}\geq s_\alpha$ then we are done since $f_{(1)}\geq(1)$.
  Otherwise one of the following cases occurs.

  \subsubsection*{Case 1: $\alpha=(p-1,1)$}
  Given $p$, there are exactly three possibilities for $\lambda$, depending on where the single cell of $\lambda/\alpha$ is added to $\alpha$.
  These possibilities are shown in Figure~\ref{figg:p1-1}.
  In this figure the cells of $\alpha=(p-1,1)$ are shown in white, while the cells of $\lambda/\alpha$ are shown in green.
  \begin{figure}
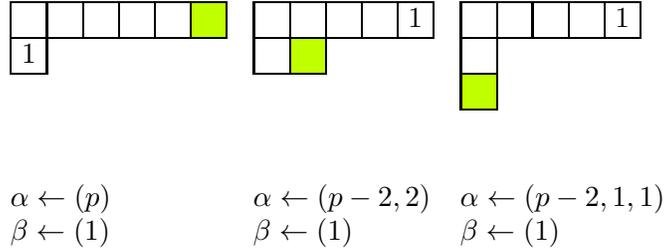

  \begin{displaymath}
      \begin{array}{llll}
        \ytableaushort{{}{}{}{}{}{},1}*[*(lime)]{5+1,0,0} & \ytableaushort{{}{}{}{}1,{}}*[*(lime)]{0,1+1,0,0} & \ytableaushort{{}{}{}{}1,{}}*[*(lime)]{0,0,1,0,0}\\
        \alpha\gets(p)& \alpha\gets(p-2,2)& \alpha\gets (p-2,1,1)\\
        \beta\gets(1)& \beta\gets(1)& \beta\gets(1)
      \end{array}
    \end{displaymath}
    \caption{Cases for $\lambda\supset (p-1,1)$}
    \label{figg:p1-1}
  \end{figure}
  \subsubsection*{Case 2: $\alpha = (2,1^{p-2})$, with $p$ odd}
  Once again, given $p$, there are finitely many possibilities for $\lambda$, as shown in Figure~\ref{figg:21p-2} along with the replacement for $\alpha$.
  \begin{figure}[h]
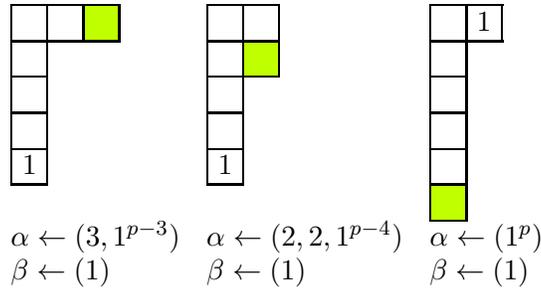

    \begin{displaymath}
      \begin{array}{llll}
        \ytableaushort{{}{},{},{},{},1}*[*(lime)]{2+1,0,0} & \ytableaushort{{}{},{},{},{},1}*[*(lime)]{0,1+1,0,0} & \ytableaushort{{}1,{},{},{},{}}*[*(lime)]{0,0,0,0,0,1}\\
        \alpha\gets(3,1^{p-3})& \alpha\gets(2,2,1^{p-4})& \alpha\gets (1^p)\\
        \beta\gets(1)& \beta\gets(1)& \beta\gets(1)

      \end{array}
    \end{displaymath}
    \caption{Cases for $\lambda\supset (2,1^{p-2})$}
    \label{figg:21p-2}
  \end{figure}     
  \subsubsection*{Case 3: $\alpha=(1^p)$ with $p$ even} 
  As before, given $p$, there is only one possibility for $\lambda\neq (1^{p+1})$ which contains $(1^p)$, as shown in Figure~\ref{figg:1p} along with the replacements for $\alpha$ and $\beta$.

  \begin{figure}[t]
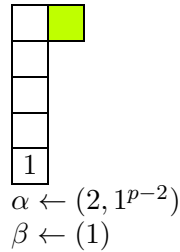

    \begin{displaymath}
      \begin{array}{llll}
        \ytableaushort{{}{},{},{},{},1}*[*(lime)]{1+1,0,0}  \\
        \alpha\gets (2,1^{p-2})\\
        \beta\gets (1)
      \end{array}
    \end{displaymath}
    \caption{Cases for $\lambda\supset (1^{p})$}
    \label{figg:1p}
  \end{figure}
  In order to prove (2), 
  choose $\beta\vdash 2$ such that $\beta \subset \lambda$ and $f_{(2)}\geq s_\beta$ using Lemma~\ref{lemma:choose-beta}.
  Choose $\alpha\vdash p$ such that $s_\alpha s_\beta\geq s_\lambda$ using Lemma~\ref{lemma:alphabeta} with the roles of $\alpha$ and $\beta$ interchanged.
  If $f_{(p)}\geq s_\alpha$ then we are done.
  Otherwise, one of the following cases occurs.
  \subsubsection*{Case 1: $\alpha=(p-1,1)$}
  Given $p$, there are finitely many possibilities for $\lambda$, depending on how the two cells of $\lambda/\alpha$ are placed.
  These possibilities are shown in Figure~\ref{figgg:p1-1}, the coloured cells having the same significance as before.
  \begin{figure}[h]
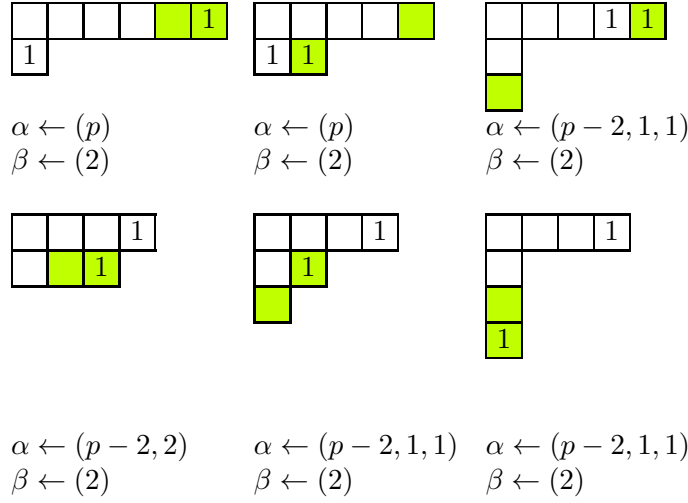

    \begin{displaymath}
      \begin{array}{lll}
        \ytableaushort{{}{}{}{}{}1,1}*[*(lime)]{4+2,0,0} & \ytableaushort{{}{}{}{}{},11}*[*(lime)]{4+1,1+1,0} & \ytableaushort{{}{}{}11,{}}*[*(lime)]{4+1,0,1} \\
        \alpha\gets(p)&\alpha\gets(p)&\alpha\gets(p-2,1,1)\\ 
        \beta\gets(2)& \beta\gets(2)& \beta\gets(2)\\
                                                         & & \\
        \ytableaushort{{}{}{}1,{}{}1}*[*(lime)]{0,1+2,0,0} & \ytableaushort{{}{}{}1,{}1}*[*(lime)]{0,1+1,1,0,0} & \ytableaushort{{}{}{}1,{},{},1}*[*(lime)]{0,0,1,1,0,0}\\
        \alpha\gets(p-2,2)&\alpha\gets(p-2,1,1)&\alpha\gets(p-2,1,1)\\
        \beta\gets(2)& \beta\gets(2)& \beta\gets(2)
      \end{array}
    \end{displaymath}
    \caption{Cases for $\lambda\supset (p-1,1)$}
    \label{figgg:p1-1}
  \end{figure}
  \subsubsection*{Case 2: $\alpha = (2,1^{p-2})$, with $p$ odd}
  Once again, given $p$, there are finitely many possibilities for $\lambda$, as shown in Figure~\ref{figgg:2p1-2} along with the replacements for $\alpha$ and $\beta$ except for the fifth diagram, where $\lambda=(2,2,1^{p-2})$.

  When $\lambda=(2,2,1^{p-2})$, it turns out that $f_{(p,2)}\ngeq s_\lambda$.
  Indeed, since $p$ is odd, $C_{(p,2)} = D_{(p,2)}$ and so $f_{(p,2)}=f_{(p)}f_{(2)}$.
  We will show that $f_{(p)}f_{(2)}\not\geq f_\lambda$ by contradiction.

  If $f_{(p)}f_{(2)}\geq f_\lambda$, there exist $\alpha\vdash p$ and $\beta\vdash 2$ such that $f_{(p)}\geq s_\alpha$, $f_{(2)}\geq s_\beta$ and $s_\alpha s_\beta\geq s_\lambda$.
  The latter condition implies that $\alpha\subset\lambda$.
  This gives us three possibilities for $\alpha$, namely $(2^2,1^{p-4})$, $(2,1^{p-2})$ and $(1^p)$.
  Since $p$ is odd, $f_{(p)}\not\geq (2,1^{p-2})$.
  But for the remaining two possibilities, $\lambda/\alpha$ is a single column, forcing $\beta=(1,1)$, but $f_{(2)}\not\geq s_\beta$, contradicting our assumptions on $\alpha$ and $\beta$.

  Hence, we have $f_{(p,2)}\ngeq s_{(2,2,1^{p-2})}$, if $p$ is odd.
  \begin{figure}[h]
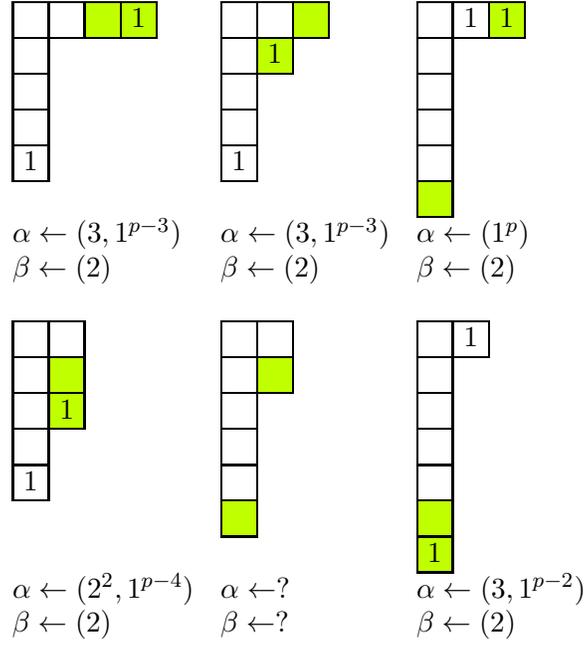

    \begin{displaymath}
      \begin{array}{lll}
        \ytableaushort{{}{}{}1,{},{},{},1}*[*(lime)]{2+2,0,0} 
        & \ytableaushort{{}{},{}1,{},{},1}*[*(lime)]{2+1,1+1,0,0}  
        & \ytableaushort{{}11,{},{},{},{},{}}*[*(lime)]{2+1,0,0,0,0,1}\\
        \alpha\gets(3,1^{p-3})
        & \alpha\gets(3,1^{p-3})
        & \alpha\gets(1^{p})\\
        \beta\gets(2)
        & \beta\gets(2)
        & \beta\gets(2)\\
        & & \\
        \ytableaushort{{}{},{},{}1,{},1}*[*(lime)]{0,1+1,1+1,0,0}
        & \ytableaushort{{}{},{}{},{},{},{},{}}*[*(lime)]{0,1+1,0,0,0,1} 
        & \ytableaushort{{}1,{},{},{},{},{},1}*[*(lime)]{0,0,0,0,0,1,1}\\
        \alpha\gets(2^2,1^{p-4})
        & \alpha\gets?
        & \alpha\gets(3,1^{p-2})\\
        \beta\gets(2)
        & \beta\gets?
        & \beta\gets(2) 
      \end{array}
    \end{displaymath}
    \caption{Cases for $\lambda\supset (2,1^{p-2})$}
    \label{figgg:2p1-2}
  \end{figure} 
  \subsubsection*{Case 3: $\alpha=(1^p)$, with $p$ even} 
  Similarly, for each $\alpha$ of this type, there are only finitely many possibilities for $\lambda$ which contains $(1^p)$, as shown in Figure~\ref{figgg:1p} along with the replacements for $\alpha$ and $\beta$, except for the third diagram (where $\lambda=(2,1^p)$) in which case we have to use a different strategy.
  \ytableausetup{boxsize=1.2em}
  \begin{figure}
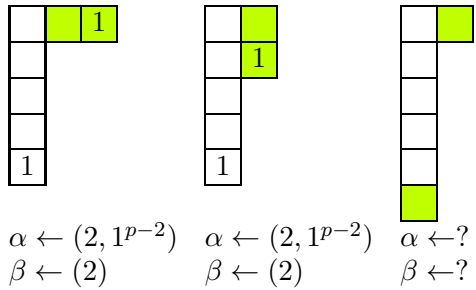

    \begin{displaymath}
      \begin{array}{llll}
        \ytableaushort{{}{}1,{},{},{},1}*[*(lime)]{1+2,0,0} & \ytableaushort{{}{},{}1,{},{},1}*[*(lime)]{1+1,1+1,0} & \ytableaushort{{}{},{},{},{},{}}*[*(lime)]{1+1,0,0,0,0,1} \\
        \alpha\gets (2,1^{p-2})& \alpha\gets (2,1^{p-2})& \alpha\gets?\\
        \beta\gets(2) & \beta\gets(2)& \beta\gets?
      \end{array}
    \end{displaymath}
    \caption{Cases for $\lambda\supset (1^{p})$}
    \label{figgg:1p}
  \end{figure}
  The restriction of the sign representation of $S_p\times S_2$ to $C_{(p,2)}$ is the trivial representation (since $w_{(p,2)}$ is an even permutation).
  Therefore, $\Ind_{C(p,2)}^{S_p\times S_2} 1$ contains the sign representation of $S_p\times S_2$.
  It follows that 
  \begin{displaymath}
    f_{(p,2)} = \ch \Ind_{C_{(p,2)}}^{S_{p+2}} \geq s_{(1^p)}s_{(1^2)} \geq s_{(2,2,1^{p-2})}
  \end{displaymath}
  Hence $f_{(p,2)}\geq s_{(2,2,1^{p-2})}$.

  In order to prove (3), 
  choose $\beta\vdash 3$ such that $\beta \subset \lambda$ and $f_{(3)}\geq s_\beta$ using Lemma~\ref{lemma:choose-beta}.
  Choose $\alpha\vdash p$ such that $s_\alpha s_\beta\geq s_\lambda$ using Lemma~\ref{lemma:alphabeta} with the roles of $\alpha$ and $\beta$ interchanged.
  If $f_{(p)}\geq s_\alpha$ then we are done.
  Otherwise, one of the following cases occurs.
  \subsubsection*{Case 1: $\alpha=(p-1,1)$}
  Given $p$, there are finitely many possibilities for $\lambda$, depending on how the three cells of $\lambda/\alpha$ are placed.
  These possibilities are shown in Figure~\ref{fig:cases}.  
  The empty cells denote the replacement for $\alpha$ and the cells filled by entries in $\{1,2,3\}$ give a skew tableau whose weight is the replacement for $\beta$ such that $f_{(p)}\geq s_\alpha$, $f_{(q)}\geq s_\beta$ and $s_\alpha s_\beta\geq s_\lambda$.

  For example, in the left-most figure of the first row, $\lambda = (p+2,1)$.
  In this case $\alpha$ is replaced by $(p)$ and $\beta$ is replaced by $(3)$.
  \begin{figure}[h]
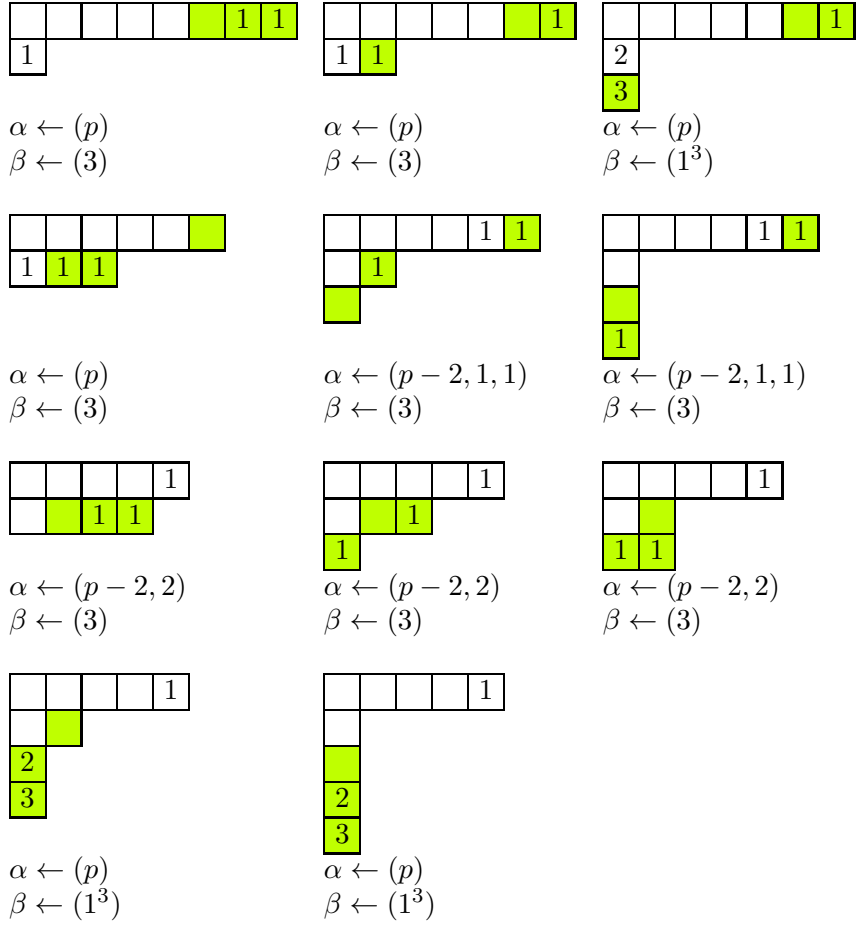

    \begin{displaymath}
      \begin{array}{lll}
        \ytableaushort{{}{}{}{}{}{}11,1}*[*(lime)]{5+3} & \ytableaushort{{}{}{}{}{}{}1,11}*[*(lime)]{5+2,1+1}& \ytableaushort{{}{}{}{}{}{}1,2,3}*[*(lime)]{5+2,0,1}\\
        \alpha\gets (p)& \alpha\gets (p)& \alpha\gets (p)\\
        \beta\gets (3) & \beta\gets (3) & \beta\gets (1^3)\\
                                                        & & \\
        \ytableaushort{{}{}{}{}{}{},111}*[*(lime)]{5+1,1+2} & \ytableaushort{{}{}{}{}11,{}1}*[*(lime)]{5+1,1+1,1} & \ytableaushort{{}{}{}{}11,{},{},1}*[*(lime)]{5+1,0,1,1}\\
        \alpha\gets (p)& \alpha\gets (p-2,1,1)& \alpha\gets (p-2,1,1)\\
        \beta\gets (3) & \beta\gets (3) & \beta\gets (3)\\
                                                        & & \\
        \ytableaushort{{}{}{}{}1,{}{}11}*[*(lime)]{0,1+3} & \ytableaushort{{}{}{}{}1,{}{}1,1}*[*(lime)]{0,1+2,1} & \ytableaushort{{}{}{}{}1,{}{},11}*[*(lime)]{0,1+1,2}\\
        \alpha\gets (p-2,2)& \alpha\gets (p-2,2)& \alpha\gets (p-2,2)\\
        \beta\gets (3) & \beta\gets (3) & \beta\gets (3)\\
                                                        & & \\ 
        \ytableaushort{{}{}{}{}1,{}{},2,3}*[*(lime)]{0,1+1,1,1} & \ytableaushort{{}{}{}{}1,{},{},2,3}*[*(lime)]{0,0,1,1,1} & \\
        \alpha\gets (p)& \alpha\gets (p)\\
        \beta\gets (1^3) & \beta\gets (1^3) \\
      \end{array}
    \end{displaymath}
    \caption{Cases for $\lambda\supset (p-1,1)$}
    \label{fig:cases}
  \end{figure}
  \subsubsection*{Case 2: $\alpha = (2,1^{p-2})$, with $p$ odd}
  Once again, for each $\alpha$ of this type, there are finitely many possibilities for $\lambda$, as shown in Figure~\ref{fig:21p-2} along with the replacements for $\alpha$ and $\beta$.
  \begin{figure}[h]
    \begin{displaymath}
      \begin{array}{lllll}
        \ytableaushort{{}{}{}11,{},{},{},1}*[*(lime)]{2+3,0,0}
        & \ytableaushort{{}{}11,{},{},{},1}*[*(lime)]{2+2,1+1,0}
        & \ytableaushort{{}111,{},{},{},{}}*[*(lime)]{2+2,0,0,0,1}
        & \ytableaushort{{}{}1,{}{},{}1,{},1}*[*(lime)]{2+1,1+1,1+1,0,0}
        &   \ytableaushort{{}{}{},{}1,{},{},2,3}*[*(lime)]{2+1,1+1,0,0,0,1}\\
        \alpha\gets (3,1^{p-3})& \alpha\gets (2^2,1^{p-4})& \alpha\gets (1^p)& \alpha\gets (2^2,1^{p-4})& \alpha\gets (3,1^{p-3}) \\
        \beta\gets (3) & \beta\gets (3) & \beta\gets (3)& \beta\gets (3)& \beta\gets (1^3)\\
        & & & & \\
        \ytableaushort{{}{}{},{},{},{},1,2,3}*[*(lime)]{2+1,0,0,0,0,1,1} & \ytableaushort{{}{},{}{},{}1,{}2,3}*[*(lime)]{0,1+1,1+1,1+1} &
                                                                                                                                          \ytableaushort{{}{},{}{},{}1,{},2,3}*[*(lime)]{0,1+1,1+1,0,0,1}
        & \ytableaushort{{}{},{}{},{},{},1,2,3}*[*(lime)]{0,1+1,0,0,0,1,1} 
        & \ytableaushort{{}1,{},{},{},{},{},2,3}*[*(lime)]{0,0,0,0,0,1,1,1}  \\
        \alpha\gets (3,1^{p-3})& \alpha\gets (2^2,1^{p-4})& \alpha\gets (2^2,1^{p-4})& \alpha\gets (2^2,1^{p-4})& \alpha\gets (1^{p}) \\
        \beta\gets (1^3) & \beta\gets (1^3) & \beta\gets (1^3)& \beta\gets (1^3)& \beta\gets (1^3)\\
      \end{array}
    \end{displaymath}
    \caption{Cases for $\lambda\supset (2,1^{p-2})$}
    \label{fig:21p-2}
  \end{figure}
  \subsubsection*{Case 3: $\alpha=(1^p)$ with $p$ even} 
  Similarly, for each $\alpha$ of this type, there are only finitely many possibilities for $\lambda$ which contains $(1^p)$, as shown in Figure~\ref{fig:1p} along with the replacements for $\alpha$ and $\beta$.
  \ytableausetup{boxsize=1.2em}
  \begin{figure}[h]
    \begin{displaymath}
      \begin{array}{lll}
        \ytableaushort{{}{}11,{},{},{},{},1}*[*(lime)]{1+3,0,0} & \ytableaushort{{}{}1,{}1,{},{},{},1}*[*(lime)]{1+2,1+1,0} & \ytableaushort{{}{}1,{},{},{},{},2,3}*[*(lime)]{1+2,0,0,0,0,0,1} \\ 
        \alpha\gets (2,1^{p-2})& \alpha\gets (2,1^{p-2})& \alpha\gets (2,1^{p-2})\\
        \beta\gets (3)& \beta\gets (3)& \beta\gets (1^3)\\
                                                                & & \\ \ytableaushort{{}{},{}1,{}2,{},{},3}*[*(lime)]{1+1,1+1,1+1} & \ytableaushort{{}{},{}1,{},{},{},2,3}*[*(lime)]{1+1,1+1,0,0,0,0,1}  & \ytableaushort{{}{},{},{},{},{},1,2,3}*[*(lime)]{1+1,0,0,0,0,0,1,1}\\
        \alpha\gets (2,1^{p-2})& \alpha\gets (2,1^{p-2})& \alpha\gets (2,1^{p-2}) \\
        \beta\gets(1^3)& \beta\gets(1^3)& \beta\gets(1^3)
      \end{array}
    \end{displaymath}
    \caption{Cases for $\lambda\supset (1^p)$}
    \label{fig:1p}
  \end{figure}
\end{proof}
\pagebreak
\begin{lemma}
  \label{lemma:simple}
  The partitions $(1,1)$, $(2,2)$ and $(3,3)$ are persistent.
\end{lemma}
\begin{proof}
  We compute the following by hand or using SageMath \cite{sagemath}.
  \begin{itemize}
  \item $f_{(1,1)}=s_{(2)}+s_{(1,1)}$
  \item $f_{(2,2)}=s_{(4)}+s_{(3,1)}+2s_{(2,2)}+s_{(2,1,1)}+s_{(1,1,1,1)} $
  \item $f_{(3,3)}=s_{(6)}+s_{(5,1)}+3s_{(4,2)}+4s_{(4,1,1)}+s_{(3,3)}+4s_{(3,2,1)}+4s_{(3,1,1,1)}+s_{(2,2,2)}+3s_{(2,2,1,1)}+s_{(2,1,1,1,1)}+s_{(1,1,1,1,1,1)}$
  \end{itemize}

  Evidently, the partitions $(1,1)$, $(2,2)$ and $(3,3)$ are persistent.
\end{proof}

Now consider partitions with three or more parts.
\begin{lemma}
  \label{lemma:power}
  A partition $\mu=(\mu_1,\dots,\mu_k)\vdash n$ is persistent if the partition $\tilde\mu$ obtained by removing a part $\mu_i$ from $\mu$ is persistent and $n-\mu_i\geq 4$.
\end{lemma}
\begin{proof}
  Since $C_{\tilde\mu}\times C_{(\mu_i)}\leq C_\mu \text{ and } D_{\tilde\mu}\times D_{(\mu_i)}= D_\mu$,
  we have
  \begin{displaymath}
    \Ind_{C_\mu}^{D_\mu}1 \geq \Ind_{C_{\tilde\mu}}^{D_{\tilde\mu}}1 \otimes \Ind_{C_{(\mu_i})}^{D_{(\mu_i)}} 1
  \end{displaymath}
  Inducing to $S_{p_1}\times\dotsb\times S_{p_k}$, and then to $S_{p_1+\dotsb+p_k}$ gives
  \begin{equation}
    \tag{*}
    f_\mu \geq f_{\tilde\mu} f_{(\mu_i)}.
  \end{equation}
  Consider $\lambda\vdash n$, $\lambda\neq (1^n)$.
  To show $f_\mu\geq s_\lambda$, it suffices to show that there exist $\alpha\vdash n-\mu_i$, $\beta\vdash \mu_i$ such that $\alpha\neq(1^{n-\mu_i})$ and $s_\alpha s_\beta\geq s_\lambda$.
  Using Lemma~\ref{lemma:choose-beta}, choose $\beta\vdash \mu_i$ such that $\beta\subset \lambda$ and $f_{(\mu_i)}\geq s_\beta$.
  Choose $\alpha$ using Lemma~\ref{lemma:alphabeta} with the roles of $\alpha$ and $\beta$ reversed.
  If $f_{\tilde\mu}\geq s_\alpha$, we are done.
  Otherwise, since $\tilde\mu$ is persistent, we must have $\alpha=(1^{n-\mu_i})$.

  Since $\lambda\neq (1^n)$, we may replace $\alpha$ by $(2,1^{n-\mu_i-2})$.
  Since $\tilde\mu$ is persistent, $f_{\tilde\mu}\geq s_\alpha$.
  Choose $\beta$ using Lemma~\ref{lemma:alphabeta}.
  If $f_{(\mu_i)}\geq s_\beta$ we are done.
  Otherwise, by Theorem~\ref{theorem:swanson}, if $\mu_i\geq 3$, $\beta$ must be one of $(\mu_i-1,1)$, $(2,1^{\mu_i-2})$ with $\mu_i$ odd, or $(1^{\mu_i})$ with $\mu_i$ even.
  We may proceed as in Case~3 of Lemma~\ref{lemma:pq} and the third part of Lemma~\ref{lemma:alphabeta1} for $\beta=(\mu_i-1,1)$ or $\beta=(2,1^{\mu_i-2})$.
  If $\mu_i=2$, then $\beta=(1,1)$ which we will consider below.
  If $\mu_i=1$, $f_{(1)}\geq s_{\beta}$ trivially.

  It remains to consider $\beta = (1^{\mu_i})$ with $\mu_i\geq 2$ and even.
  In this case, $\lambda/\alpha$ has just one column, so $\lambda$ must be $(2,1^{n-2})$.
  Since $f_{\tilde\mu}\ngeq s_{(1^{n-\mu_i})}$, the permutation $w_{\tilde\mu}$ has to be an odd permutation, i.e., an odd number of the parts $\mu_j$, $j\neq i$ is even.
  In particular, $\mu_j$ is even for some $j\neq i$.
  Let $\hat\mu$ denote the composition obtained from $\tilde\mu$ by deleting $\mu_j$.
  We have:
  \begin{displaymath}
    f_\mu \geq f_{\hat\mu}f_{(\mu_j,\mu_i)}.
  \end{displaymath}
  But both $w_{\hat\mu}$ and $w_{(\mu_j,\mu_i)}$ are even permutations, so
  \begin{displaymath}
    f_{\hat\mu}\geq s_{(1^{n-\mu_i-\mu_j})} \text{ and } f_{(\mu_j,\mu_i)}\geq s_{(1^{\mu_j+\mu_i})}.
  \end{displaymath}
  So $f_\mu \geq s_{(1^{n-\mu_i-\mu_j})}s_{(1^{\mu_j+\mu_i})} \geq s_{(2,1^{n-2})} = s_\lambda$.
\end{proof}
\begin{proof}[Proof of the Main Theorem]
  For $\lambda=(1^n)$ we know that $f_\mu\geq s_\lambda$ if and only if $w_\mu$ is an even permutation.
  If $\lambda\neq (1^n)$, consider the following cases.

  \subsubsection*{Case 1: $\mu$ has only one part}
  Theorem~\ref{theorem:swanson} tells us that $w_\mu$ has an invariant vector in $V_\lambda$ if and only if $\lambda$ is not of the form $(n-1,1)$, $(1^n)$, or $(2,1^{n-2})$ with $n$ odd, which are cases 2 and 3 of the main theorem. 

  \subsubsection*{Case 2: $\mu$ has two parts}
  By Lemmas~\ref{lemma:pq} and~\ref{lemma:alphabeta1} $f_{(p,q)}\geq s_\lambda$ unless one of the following happens.
  \begin{itemize}
  \item $\lambda = (2,2,1^{n-4})$ and $(p,q) = (n-2,2)$ with $n$ (corresponding to case~4 of the main theorem).
  \item $(p,q)\in \{(4,3),(5,3),(3,1)\}$.
    Checking by direct calculation will imply that $f_{(p,q)}\geq s_\lambda$ unless $(p,q)=(3,1)$ and $\lambda=(2,2)$ or $(p,q) = (5,3)$ and $\lambda = (2^4)$ or $(4,4)$ corresponding to cases 5, 7, and 8 of the main theorem. 
  \end{itemize}
  \subsubsection*{Case 3: $\mu$ has at least three parts}
  Assume that $n\geq 11$.

  If $\mu_2\geq 4$ then $(\mu_1,\mu_2)$ is persistent by Lemma~\ref{lemma:pq}.
  Using  Lemma~\ref{lemma:power} repeatedly, we can show that $\mu$ is persistent.

  If $\mu_1\geq 4$ and $\mu_2<4$ and either 2 or 3 occurs at least twice in $\mu$, then using the Lemma~\ref{lemma:simple} and Lemma~\ref{lemma:power} repeatedly, $\mu$ is persistent.

  Otherwise $\mu_i\in\{1,2,3\}$ for $i\geq 2$, with $2$ and $3$ not repeated.
  If $1$ occurs in $\mu$, since $\mu_1\geq 4$, $(\mu_1,1)$ is persistent by Lemma~\ref{lemma:alphabeta1}, so $\mu$ is persistent by repeated use of Lemma~\ref{lemma:power}.
  If $1$ does not occur in $\mu$, then $\mu_1\geq 6$ and $\mu=(\mu_1,3,2)$ (since $\mu$ has at least three parts).
  By Lemma~\ref{lemma:alphabeta1} $(\mu_1,3)$ is persistent.
  Hence $\mu$ is persistent by repeated use of Lemma~\ref{lemma:power}.

  Finally, suppose $\mu_1\leq3$.
  Once again, if either $2$ or $3$ occurs at least twice then $\mu$ is persistent by the Lemmas~\ref{lemma:simple} and ~\ref{lemma:power}.
  Otherwise, $1$ occurs 6 times.
  Since $f_{(1^6)}$ is the Frobenius characteristic of the regular representation of $S_6$, it is persistent.
  Hence $\mu$ is persistent by Lemma~\ref{lemma:power}.

  If $n<11$, a direct computation using SageMath \cite{sagemath} tells us that $\mu=(3,2,1)$ and $\lambda=(2^3)$ (corresponding to case 6) and $\mu=(5,3,2)$ and $\lambda=(2^5)$ (corresponding to case 9) are the only pairs which violate $f_\mu\geq s_\lambda$.

  This completes the proof of the main theorem.
\end{proof}
\subsection*{Acknowledgements}
We thank Arvind Ayyer for bringing this problem to our attention and for fruitful discussions.
We thank Dipendra Prasad and S Viswanath for their help and encouragement.
We thank Sheila Sundaram and Joshua Swanson for their comments on a first draft of this article.
We also thank Jyotirmoy Ganguly, Rijubrata Kundu, Digjoy Paul, and Papi Ray for generously sharing their insights to this problem.
\bibliographystyle{abbrv}
\bibliography{refs}

\end{document}